\documentclass[a4paper,10pt,reqno, english]{amsart}

\usepackage{amsmath,amssymb,amscd,amsthm,amsfonts}
\usepackage{graphicx,subfigure}
\usepackage{hyperref}
\usepackage{dsfont}
\usepackage[nobysame, alphabetic]{amsrefs}
\usepackage{tikz}
\usepackage[capitalise]{cleveref}
\usepackage{mathrsfs}
\usepackage{booktabs,array}
\usepackage{microtype}
\usepackage{enumitem}

\newtheorem{theorem}{Theorem}
\newtheorem{lemma}{Lemma}

\newtheorem{conjecture}{Conjecture}

\crefname{conjecture}{Conjecture}{Conjectures}
\Crefname{conjecture}{Conjecture}{Conjectures}

\def\rr{\mathds{R}}

\DeclareMathOperator{\conv}{conv}

\title{The topological B\'ar\'any--Larman conjecture for prime numbers}

\hypersetup{
  pdftitle={The topological B\'ar\'any--Larman conjecture for prime numbers},
  pdfauthor={Pablo Soberon}
}

\author[Sober\'on]{Pablo Sober\'on}\address{Baruch College \& The Graduate Center, City University of New York, One Bernard Baruch Way, New York, NY 10010, United States} 
\email{psoberon@gc.cuny.edu}

\thanks{
The research of P. Sober\'on is supported by NSF CAREER grant DMS-2237324 and a PSC-CUNY Track 1 award.}

\keywords{B\'ar\'any--Larman conjecture; colorful Tverberg theorem; topological Tverberg theorem; computer-assisted proof.}

\subjclass[2020]{Primary 52A35; Secondary 68V05}

\begin{document}

\begin{abstract}
The B\'ar\'any--Larman conjecture states that for any $d+1$ sets of $r$ points each in $\mathbb{R}^d$, considered as color classes, we can partition their union into $r$ rainbow $(d+1)$-tuples whose convex hulls intersect.  We prove that the topological version of this conjecture holds when $r$ is a prime number.  We also show that the optimal colorful Tverberg theorem of Blagojevi\'c, Matschke, and Ziegler cannot be extended to prime powers.
\end{abstract}

\maketitle

\section{Introduction}

In $1992$, B\'ar\'any and Larman made a beautiful conjecture regarding a colorful version of Tverberg's theorem.  Given a set $X$ in $\rr^d$, we denote by $\conv X$ its convex hull.

\begin{conjecture}[B\'ar\'any, Larman 1992 \cite{Barany1992}]\label{conj:baranylarman}
    Let $r,d$ be positive integers.  Suppose we are given $d+1$ sets $X_1,\dots,X_{d+1}$, each with $r$ points in $\rr^d$.  Then, there is a partition of their union into $r$ sets $C_1,\cdots,C_r$ such that $|C_j \cap X_i| = 1$ for all $i,j$, and $\bigcap_{j=1}^r\conv C_j \neq \emptyset$.
\end{conjecture}

The purpose of this manuscript is to prove the topological version of \cref{conj:baranylarman} when $r$ is a prime number.  Namely, we prove the following.

\begin{theorem}\label{thm:main}
    Let $p$ be a prime number and $d$ be a positive integer.  Let $A_1,\dots,A_{d+1}$ be sets of $p$ points each.  Consider the simplicial complex $K = A_1 * \dots * A_{d+1}$.  Then, for every continuous function $f: K\to\rr^d$, there exist points $x_1,\dots,x_p$ in pairwise disjoint faces of $K$ such that $f(x_1)=\dots=f(x_{p})$.
\end{theorem}

This implies \cref{conj:baranylarman} when $r$ is a prime number by taking $r=p$, mapping the sets $A_i$ to $X_i$ for each $i$, and extending $f$ linearly.

\cref{conj:baranylarman} is known as the colorful Tverberg conjecture.  It is called colorful as we can think of each $X_i$ as a set of points of a different color.  The conclusion we seek is to find $r$ rainbow simplices whose convex hulls intersect.  This is related to Tverberg's theorem, which says that \textit{any set of $(r-1)(d+1)+1$ points in $\rr^d$ has a partition into $r$ parts whose convex hulls intersect} \cite{Tverberg1966}.  Tverberg's theorem is a central result in discrete and combinatorial geometry; this result and its variants have motivated significant progress in the field \cites{Barany2018, Blagojevic2017, Matousek2002}.

When B\'ar\'any and Larman made the conjecture, they proved the case $d=2$ (any $r$), and presented Lov\'asz's proof of the case $r=2$ (any $d$), also known as the colorful Radon theorem \cref{conj:baranylarman}.

\cref{conj:baranylarman} is closely related to more general colorful Tverberg partitions, in which each set $X_i$ is allowed to have more than $r$ points.  The first general existence theorem for this kind of partitions was shown by \v{Z}ivaljevi\'c and Vre\'{c}ica \cite{Zivaljevic1992}.  Since then, there have been many results around colorful Tverberg partitions, including proofs with more elementary arguments \cites{Matousek2012, ludwigson2026helly}, variations with more color classes \cites{Soberon2015, Sarkar2022}, changing the conditions on the parameters \cite{Mauri2026}, and the constraints method \cite{Blagojevic2014}.

One of the most striking advances around \cref{conj:baranylarman} is due to Blagojevi\'c, Matschke, and Ziegler, who solved the case then $r+1$ is a prime number (any $d$) \cites{Blagojevic2015, Blagojevic2011} in the affirmative.  The case when $r+1$ is prime is a corollary of their optimal colorful Tverberg theorem.  Their proof also implies the topological version of \cref{conj:baranylarman} when $r+1$ is prime.  In particular, their main result has the following consequence.

\begin{theorem}[Blagojevi\'c, Matschke, Ziegler 2015]\label{thm:BMZ-optimal}
    Let $p$ be a prime number and $d$ be a positive integer.  Let $A_1,\dots, A_{d+1}$ be sets of $p-1$ elements each, and $a$ be an additional element.  Consider the simplicial complex $K = A_1*\dots A_{d+1}*\{a\}$.  Then, for any continuous function $f:K \to \rr^d$, there exist $p$ points $x_1,\dots,x_{p}$ from pairwise disjoint faces of $K$ such that $f(x_1)=\dots=f(x_{p})$.
\end{theorem}

To prove the case of \cref{conj:baranylarman} when $r+1$ is prime, it suffices to add an additional point $a$ in $\rr^d$, construct $f:K \to \rr^d$ linearly as before, apply \cref{thm:BMZ-optimal} with $p=r+1$, and ignore the part that contains $a$.  To our surprise, the case when $r$ is prime is also a direct consequence of \cref{thm:BMZ-optimal}.  We present the proof of \cref{thm:main} in \cref{sec:main}

With the results presented in this note, the first open case of \cref{conj:baranylarman} is $r=8, d=3$.  To prove more cases of \cref{conj:baranylarman}, it is tempting to try to generalize \cref{thm:BMZ-optimal} when $p$ is not prime.  A version of \cref{thm:BMZ-optimal} without any conditions on the parameters is impossible, as it would imply the topological Tverberg theorem when the number of parts is not a prime power, which is known to be impossible \cites{Frick:2015wp, Blagojevic2019}.

In \cref{sec:counterexample}, we show that \cref{thm:BMZ-optimal} does not hold for prime powers, even if $f$ is an affine map.  Note that the counterexample does not contradict the results of Joji\'c, Panina, and \v{Z}ivaljevi\'c \cite{Jojic2022}, since their extensions of \cref{thm:BMZ-optimal} require prescribed multiplicities.  Our example disproves the ``colorful Tverberg-Vre\'cica'' conjecture of Blagojevi\'c, Matschke, and Ziegler \cite{Blagojevic2011}*{Conj. 1.2} with parameters $k=0$, $r_0 = 4$, and $d\ge 3$.

 \section{Proof of \cref{thm:main}}\label{sec:main}

Suppose we are given sets $A_1,\dots,A_{d+1}$, each with $p$ elements.  We think of each $A_i$ as a color class, so that the faces of $K= A_1 * \dots * A_{d+1}$ correspond to rainbow sets.  We will say that $A_{d+1}$ is red.  To use \cref{thm:BMZ-optimal}, we only need sets with $p-1$ elements.  Our main strategy will be to remove one point from each set and then construct a smaller simplicial complex.  However, since we need to include a join with one more singleton set, we will allow our sets to have two red points.  The bulk of the proof is showing that we can fix faces that use two red points.

Formally, consider elements $a_1 \in A_1, \dots, a_{d+1}\in A_{d+1}$ and let $A'_i = A_i \setminus \{a_i\}$.  Now consider the simplicial complex $K' = A'_1 * \dots *A'_{d+1}*\{a_{d+1}\}$.  We can apply \cref{thm:BMZ-optimal} to $K'$, but first we need to extend $f$ to this complex.

For a simplicial complex $T$ and a set of vertices $S$ of $T$, let $T[S]$ be the set restriction of $T$ to $S$.

First we apply a standard dimension-counting argument to show that \cref{thm:BMZ-optimal} holds for maps defined on the $d$-dimensional skeleton of the domain.  We include the proof for completeness.

\begin{lemma}\label{lem:d-skeleton}
    Let $p$ be a prime number and $d$ be a positive integer.  Let $A_1,\dots, A_{d+1}$ be sets of $p-1$ elements each, and $a$ be an additional element.  Consider the simplicial complex $K = A_1*\dots A_{d+1}*\{a\}$ and $L$ be the $d$-dimensional skeleton of $K$.  Then, for any continuous function $h:L \to \rr^d$, there exist $p$ points $x_1,\dots,x_{p}$ from pairwise disjoint faces of $L$ such that $h(x_1)=\dots=h(x_{p})$.
\end{lemma}

\begin{proof}
    Given $h:L\to \rr^d$, we can extend it to a map $\tilde{h}:K \to \rr^d$.  We can also approximate $\tilde{h}$ piecewise-linear maps $H$ in general position.  Apply \cref{thm:BMZ-optimal} to $H$, and let $\sigma_1,\dots, \sigma_p$ be the resulting pairwise-disjoint faces whose images intersect.

    If any of those faces, say $\sigma_1$ has dimension $d+1$, then using the fact that the total number of vertices is $(p-1)(d+1)+1$ we have that
    \[
    \sum_{i=2}^p \dim \sigma_i \le (p-1)(d+1)+1 - (d+2) - (p-1) < (p-2)d.
    \]
    However, if the intersection of $p-1$ simplices in general position is non-empty, their dimensions must add to at least $(p-2)d$, a contradiction.  A standard limiting argument finishes the proof.
\end{proof}

The following is the main technical lemma for the proof.

\begin{lemma}\label{lem:fixit}
    Let $p$ be a prime number, $d$ be a positive integer, and $A_1,\dots,A_{d+1}$ be sets of $p$ points each.  Let $a_1\in A_1, \dots, a_{d+1}\in A_{d+1}$ be elements, and $A'_i = A_i \setminus \{a_i\}$ for $i=1,\dots,d+1$.  Consider the simplicial complexes $K=A_1 * \dots * A_{d+1}$, $K'=A'_1 *\dots * A'_{d+1}*\{a_{d+1}\}$, and let $L$ be the $d$-skeleton of $K'$.  Then, there exists a continuous map $g: L \to K$ that is the identity on $K \cap L$ and such that for every face $\sigma$ of $L$ we have $g(\sigma) \subseteq K[\sigma \cup \{a_1,\dots,a_{d}\}]$.
\end{lemma}

\begin{proof}
    The vertices $a_1,\dots,a_d$ do not belong to $L$.  A face of $L$ is not a face of $K$ precisely when it contains two vertices of $A_{d+1}$, one of which is $a_{d+1}$.  We call the $\sigma$ with two red vertices the bad faces.

    We define $g$ on $K \cap L$ as the identity, and then proceed inductively on the dimension of the bad faces.  Suppose that $\sigma$ is a bad face of $L$ and for every proper subface $\tau \subset \sigma$, we have defined $g$ on $\tau$ inductively.  A bad face $\sigma$ has at most $d+1$ vertices, and two of them are red.  Therefore, it misses at least one color (without loss of generality, assume it misses $A_1$).  
    
    Consider the complex $M_{\sigma}=K[\sigma \cup \{a_1,\dots,a_{d}\}]$.  Notice that $a_1$ is the only blue vertex of $M_{\sigma}$, so $M_{\sigma}$ is a cone with vertex $a_1$ and is therefore contractible.

    The map $g$ is already defined on $\partial{\sigma}$.  Moreover, every proper face $\tau \subset \sigma$ is either fixed by $g$ (if it misses a red vertex of $\sigma$), or is it mapped to $M_\tau \subset M_{\sigma}$.  We can cone this boundary map to $a_1$ to extend $g$ over $\sigma$.
\end{proof}

We are now ready to prove \cref{thm:main}.

\begin{proof}[Proof of \cref{thm:main}]
    Let $K, K'$, $L$, and $a_1,\dots,a_{d+1}$ be as in \cref{lem:fixit}.  Consider a map $g:L \to K$ from \cref{lem:fixit}.  Starting with our continuous map $f: K \to \rr^d$, we can construct $h = f \circ g: L \to \rr^d$.  We can apply \cref{lem:d-skeleton} to $h$.
    
    In other words, we can find $x_1,\dots,x_p \in L$ points from pairwise disjoint faces $\sigma_1,\dots,\sigma_p$ of $L$ such that $h(x_1) = \dots = h(x_p)$.  If $\sigma_i$ contains at most one red vertex, it is a face of $K$, $x_i \in K$, and $h(x_i) = f(x_i)$.  Therefore, if none of the face $\sigma_1,\dots,\sigma_p$ have two red vertices, we are done.  
    
    Otherwise, there is at most one face $\sigma_i$ that contains two red vertices, as it must have $a_{d+1}$.  In this case $h(\sigma_i) = f(g(\sigma_i))$, and $g(\sigma_i) \subseteq K[\sigma_i\cup\{a_1,\dots,a_d\}]$.  Let $\tau_i$ be a face of $K[\sigma_i\cup\{a_1,\dots,a_d\}]$ that contains $g(x_i)$.  The disjointness condition implies that $\tau_i$ is disjoint from any face $\sigma_j$ with $j \neq i$.  The points $x_j$ with $j \neq i$ and $g(x_i)$ satisfy the conclusion of the theorem.

    If the selected rainbow faces of $K$ containing $x_1,\dots,x_p$ do not use every color, they can be extended to satisfy that condition without losing disjointness.
\end{proof}

 \section{Counterexample to the colorful affine optimal Tverberg theorem for prime powers}\label{sec:counterexample}

With the current methods, any new pairs $(d,p)$ for which \cref{thm:BMZ-optimal} holds would imply \cref{conj:baranylarman} for $r=p$ and $r=p-1$. Due to the counterexamples for the topological Tverberg theorem, we know that $p$ must be a prime power, unless $d$ is sufficiently small.  However, even for $p$ a prime power the result may fail with affine maps.  We disprove the case $p=4, d\ge 3$.  The affine case for $p=4,d=2$ has been proved by Kliem \cite{kliem2021new}.

\subsection{Dimension $3$, $r=4$}

Below are the coordinates of $13$ points ($x_0$ through $x_{12}$) in $\rr^3$, presented as four triples and a singleton.  If we consider each triple as a set of a different color, no colorful Tverberg partition into four parts exists.  The example was found with the use of an LLM.

 \begin{table}[htbp]
\centering
\begin{tabular}{ccl}
\toprule
Index & Color & Coordinates    \\
\midrule
$0$ & $X_1$ & ($13$,  $60$,  $-64$)\\
$1$ & $X_1$ & ($355$,  $256$,  $429$)\\
$2$ & $X_1$ & ($-85$,  $-35$,  $-309$)\\
\addlinespace
$3$ & $X_2$ & ($17$,  $59$,  $-62$)\\
$4$ & $X_2$ & ($174$,  $-631$,  $165$)\\
$5$ & $X_2$ & ($18$,  $61$,  $-63$)\\
\addlinespace
$6$ & $X_3$ & ($12$,  $-114$,  $-79$)\\
$7$ & $X_3$ & ($95$,  $132$,  $14$)\\
$8$ & $X_3$ & ($-583$,  $0$,  $371$)\\
\addlinespace
$9$ & $X_4$ & ($14$,  $80$,  $-66$)\\
$10$ & $X_4$ & ($17$,  $60$,  $-61$)\\
$11$ & $X_4$ & ($25$,  $56$,  $-64$)\\
\addlinespace
$12$ & $X_5$ & ($-72$,  $17$,  $-212$)\\
\bottomrule
\end{tabular}
\caption{Coordinates of points for the affine counterexample for $r=4, d=3$}\label{tab:coordinates}
\end{table}

To verify that this is indeed a counterexample, as simple brute-force approach is sufficient.  A standard technique to determine if a partition of a set of points in $\rr^d$ is a Tverberg partition follows from Sarkaria's proof of Tverberg's theorem and its simplification by B\'ar\'any and Onn \cites{Sarkaria1992, Barany1996}.

Given a partition $B_1,\dots, B_r$ of points in $\rr^d$, to verify that it is a Tverberg partition, we can do the following.  First, let $v_1,\dots, v_r$ be the vertices of a non-degenerate simplex in $\rr^{r-1}$ whose interior contains the origin.  Then, for each $j=1,\dots,r$ and each point $b \in B_j$, take $(b,1)\otimes v_j \in \rr^{(d+1)(r-1)}$.  The resulting set in $\rr^{(d+1)(r-1)}$ contains the origin in its convex hull if and only if the original partition was a Tverberg partition.

So, for each rainbow partition of our set of $13$ points, we can use Sarkaria's technique to lift them to $\rr^{12}$.  This will result in $13$ points with a unique linear dependence (this is also confirmed by the program), and it suffices to check that the non-trivial linear dependence has both positive and negative coefficients to rule out that the convex hull of the set of points contains the origin.  This is done exhaustively by a Python program that it can be found in the repository below.  The Python program checks each of the $55{,}296=4 \cdot 24^3$ possible partitions.  We can fix the partition of one of the triples in advance.  Then, we want to assign labels in $\{1,2,3,4\}$ to each point so that no color repeats labels, we have four options for the singleton point and $24$ options for each of the remaining three triples.  Our Python program uses $v_1 = (1,0,0)$, $v_2=(0,1,0)$, $v_3=(0,0,1)$, and $v_4=(-1,-1,-1)$.

The following repository contains the Python program that verifies the example and a \texttt{README} file that explains it in greater detail.

\vspace{0.5cm}

\textbf{Repository: }

\url{https://github.com/psoberon-math/barany-larman-counterexample-verification}

\subsection{Dimensions greater than $3$}

We prove the following lemma, extending our example to higher dimesions.

\begin{lemma}
    Let $d \ge 3$ be an integer.  There exists a set of $3d+4$ points in $\rr^d$ colored with $d+1$ color classes with three element and a color class with a single element that has not colorful Tverberg partition into four part.
\end{lemma}

\begin{proof}
    We proceed by induction on $d$.  The case $d=3$ is the computer-assisted example above.  Now assume that the example in $\rr^d$ has been constructed.  
    
    Embed the $d$-dimensional example in $\rr^{d+1}$ as a set of points whose last coordinate is $0$.  Then, include a set of three points of a new color whose last coordinate is $1$ (call these points blue).

    If the new set has a colorful Tverberg partition into four parts, at least one part has no blue points.  Therefore, the point $x$ of intersection of the convex hulls must have last coordinate equal to $0$.  Moreover, any convex combination of a part that gives $x$ must assign coefficient $0$ to the blue points, or the last coordinate of the resulting point would be positive.  In other words, if we remove the blue points we have a colorful Tverberg partition of the example in $\rr^d$ into four parts, which by induction does not exist.
\end{proof}

\subsection*{Disclosure of AI use}  An LLM was used to simplify the proof of the main result and to construct the counterexample.  The author has reviewed all arguments used, and has executed and reviewed the verification program.  The author assumes all responsibility for the results presented in this manuscript.
\bibliography{refref}

\end{document}